\documentclass[12pt,letterpaper]{article}
\usepackage{graphicx}
\usepackage{amsmath}
\usepackage{amsfonts}
\usepackage{amsthm}
\usepackage{amssymb}
\usepackage{color}

\usepackage{verbatim} 
\usepackage{hyperref}
\usepackage{breakurl}

\newtheorem{prop}{Proposition}
\newtheorem{corollary}[prop]{Corollary}
\newtheorem{theorem}[prop]{Theorem}
\newtheorem{lemma}[prop]{Lemma}

\theoremstyle{remark}

\newcommand{\be}{\begin{equation}}
\newcommand{\ee}{\end{equation}}
\newcommand{\bes}{\begin{equation*}}
\newcommand{\ees}{\end{equation*}}
\newcommand{\bea}{\begin{eqnarray}}
\newcommand{\eea}{\end{eqnarray}}
\newcommand{\beas}{\begin{eqnarray*}}
\newcommand{\eeas}{\end{eqnarray*}}

\begin{document}

\title{Parallel Weighings}

\author{Tanya Khovanova\\MIT}

\maketitle

\begin{abstract}
I introduce, solve and generalize a new coin puzzle that involves parallel weighings.
\end{abstract}

\section{Ancient Coin Problems}\label{ancient}

I heard my first coin problem when I was very young: 

\begin{quote}
Given 9 coins, one of them fake and lighter, find the fake coin in two weighings on a balance scale.
\end{quote}

I believed that this problem was thousands of years old, that Pythagoras could have invented it. But surprisingly its first publication was by E. D. Schell in the January 1945 issue of the American Mathematical Monthly \cite{Schell}. This does not prove that Pythagoras did not invent it, but it makes the said event highly unlikely.

I will not be surprised if everyone who reads this paper has heard this problem before and knows how to solve it. But I still need to discuss it to establish the methods that are used later in my featured problem.

First, the problem implies that all real coins weigh the same, and that we need to find a strategy that guarantees finding the fake coin in two weighings. 

By the way, two weighings is the smallest number of weighings that guarantee finding the fake coin. In addition, 9 is the largest number of coins such that the fake coin can be found in two weighings. Rather than discussing the solution for the 9-coin problem, I would like to generalize it to any number of coins:

\begin{quote}
Given $N$ coins, one of them fake and lighter, find the minimum number of weighings that will guarantee finding the fake coin.
\end{quote}

Here is a way to think about it. In one weighing divide all coins into three piles: the coins that go onto the left pan, the coins that go onto the right pan, and the coins that do not go on the scale. Clearly, we need to put the same number of coins on the pans, otherwise, we do not get any meaningful information. If the scale balances then the fake coin is in the leftover pile. If the scale does not balance, then the fake coin is in the pile that is lighter. Either way, the coin is in one of three piles and we have to start all over with this smaller pile. So to minimize the number of weighings, we need to divide all the coins into three piles so that the largest pile is the smallest. Oops, that didn't sound right. I mean, we need to minimize the size of the largest pile. Dividing the coins into three piles as evenly as possible allows us with $n$ weighings to find the fake coin among up to $3^n$ coins. In particular, we can find the fake coin among 9 coins in two weighings.

Another famous coin puzzle appeared almost at the same time as the previous puzzle \cite{Eves}.

\begin{quote}
There are 12 coins; one of them is fake. All real coins weigh the same. The fake coin is either lighter or heavier than the real coins. Find the fake coin and figure out whether it is heavier or lighter in 3 weighings on a balance scale.
\end{quote}

The solution is well known and quite beautiful. Unsurprisingly it generated more publications than the 9-coin problem \cite{BDescartes}, \cite{Goodstein}, \cite{Grossman}, \cite{Withington}. Readers who do not know the solution should try it.

What is the minimum number of weighings in this puzzle's setting for any number of coins? I just want to point out that we need to assume that the number of coins is more than 2, otherwise we cannot solve it at all.

If there are $N$ coins, then there are $2N$ possible answers to this puzzle. We need to pinpoint the fake coin and say whether it is heavier or lighter. Each weighing divides information into three parts, so in $n$ weighings we can give $3^n$ different answers. Thus, the expected number of weighings should be of the order $\log_32N$. The exact answer can be calculated using this additional constraint of having the same number of coins on each pan in each weighing. The exact answer is $(3^n-3)/2$, see \cite{Dyson}, \cite{Fine}.

In the following important variation of the latter puzzle we need to find the fake coin, but do not need to tell whether is it heavier or lighter \cite{Dyson}.

\begin{quote}
There are $N$ coins; one of them is fake. All real coins weigh the same. The fake coin is either lighter or heavier than the real coins. What is the maximum number of coins for which you can guarantee finding the fake coin with $n$ weighings on a balance scale?
\end{quote}

This problem is very similar to the previous one. Let me call the previous problem the \textit{find-and-label} problem, as opposed to this problem that I will call the \textit{just-find} problem.

The answer to the just-find problem is $(3^n-1)/2$, see \cite{Mauldon}. In particular, 13 coins is the best we can do in 3 weighings. 

Notice that for every strategy for the find-and-label problem that resolves $n$ coins, we can produce a strategy for the just-find problem that resolves $n+1$ coins by adding a coin that is never on the scale. Indeed, in the find-and-label strategy by the last weighing at least one of the weighings needs to be unbalanced to label the fake coin. Thus, if all the weighings balance at the end, the fake coin is the extra coin.

\section{The Original Parallel Weighings Puzzle}\label{sec:original}

We have all been hearing about parallel computing, and now it has turned up in a coin-weighing puzzle invented by Konstantin Knop. The puzzle appeared at 2012 Russia-Ukraine Puzzle Tournament \cite{URTour} and in Konstantin Knop's blog \cite{Knop}.

\begin{quote}
We have $N$ indistinguishable coins. One of them is fake, and it is not known whether it is heavier or lighter than the genuine coins, which all weigh the same. There are two balance scales that can be used in parallel. Each weighing lasts one minute. What is the largest number of coins $N$ for which it is possible to find the fake coin in five minutes?
\end{quote}

\section{The Road Map}

Section~\ref{sec:multiple} describes the similarity of the original puzzle with a multiple-pans problem: a coin weighing puzzle involving balance scales with not two, but any number of pans. The notion of a coin's potential---a useful technical tool in solving coin weighing puzzles---is defined in Section~\ref{sec:potential}.  How many coins with known potential can be processed in $n$ minutes is discussed there too. Section~\ref{sec:unlimited} provides a solution to the parallel weighing problem in case we have an unlimited supply of real coins. It is followed by a solution to the original puzzle and its generalization for any number of minutes in Section~\ref{sec:solution}. Section~\ref{sec:morescales} generalizes these results to the use of more than two scales in parallel. The find-and-label variation of this problem for any number of minutes is discussed in Section~\ref{sec:findlabel}. The last Section~\ref{sec:lazy} compares the find-and-label problem with the just-find problem.

\section{Warm-up: Multiple Pans Problem}\label{sec:multiple}

Knop's puzzle reminds me of another coin-weighing problem, where in a similar situation you need to find a fake coin by using five weighings on one scale with four pans. The answer in this variation would be $5^5 = 3125$. Divide coins in five groups with the same number of coins and put four groups on the four pans of the scale. If one of the pans is different (heavier or lighter), then this pan contains the fake coin. As it is one out of four pans, then after the weighing we will know the deviation of the fake coin. Otherwise, the leftover group contains the fake coin. The strategy is to divide the coins into five piles as evenly as possible. This way each weighing reduces the pile with the fake coin by a factor of five. Thus, it is possible to resolve $5^n$ coins in $n$ weighings. 

I leave it to the reader to check that, excluding the case of two coins, any number of coins greater than $5^{n-1}$ and not greater than $5^n$ can be optimally resolved in $n$ weighings.

One scale with four pans gives you more information than two scales with two pans used in parallel. We can expect that Knop's puzzle requires at least the same number of weighings as the four-pan puzzle for the same number of coins. So the answer to Knop's puzzle should not not exceed 3125. But what will it be?

If you know Russian, you can read the author's solution to the original puzzle at Knop's blog \cite{Knop}, otherwise, bear with me and you will get the answer to this puzzle, as well as the answers to this puzzle's generalizations.

\section{Coins Potential}\label{sec:potential}

While weighing coins, we may be able to determine some incomplete information about a coin's reality.  For instance, we may be able to rule out the possibility that a given coin is fake-and-heavy, without being able to tell whether that coin is real or fake-and-light.  Let us call such a coin \textit{potentially light}; and conversely, let us say a coin is \textit{potentially heavy} if it could be real or fake-and-heavy but cannot be fake-and-light.

How many coins with known potential can be processed in $n$ minutes?

If all the coins are potentially light then we can find the fake coin out of $5^n$ coins in $n$ minutes. Indeed, in this case using two scales or one scale with four pans (see Section~\ref{sec:multiple}) gives us the same information. The pan that is lighter contain the fake coin. If everything balances, then the fake coin is not on the scales.

What if there is a mixture of potentials? Can we expect the same answer? How much more complicated could it be? Suppose there are five coins: two of them are potentially light and three are potentially heavy. Then on the first scale we compare one potentially light coin with the other such coin. On the other scale we compare one potentially heavy coin against another potentially heavy coin. The fake coin can be determined in one minute.

Our intuition suggests that it is a bad idea to compare a potentially heavy coin on one pan with a potentially light coin on the other pan. Such a weighing, if unbalanced, will not produce any new information. On the other hand, if we compare a potentially heavy coin with a potentially heavy coin, then we will get new information. If the scale balances, then both coins are real. If the scale does not balance, then the fake coin is the heavier coin out of the two that are potentially heavy.

Does this mean that we should only put coins with the same potential on the same scale? Actually, we can mix the coins. For example, suppose we put 3 potentially light coins and 5 potentially heavy coins on each pan of the same scale. If the left pan is lighter, then the potentially heavy coins on the left pan and potentially light coins on the right pan must be genuine. The fake coin must be either one of the three potentially light coins on the left pan or one of the five potentially heavy coins on the right pan.

In general, after each minute, the best hope is to have the number of coins that are not determined to be real to be reduced by a factor of 5. If one of the weighings on one scale is unbalanced, then the potentially light coins on the lighter pan, plus the potentially heavy coins on the heavier pan would contain the fake coin. We do not want this number to be bigger than one-fifth of the total number of coins being processed. So, divide coins in pairs with the same potential and from each pair put the coins on different pans of the same scale. 

In one minute we can divide the group into five equal, or almost equal, groups. If there is an odd number of coins with the same potential, then the extra coin does not go on the scales. The only thing left to check is what happens if the number of coins is small. Namely, we need to check what happens when the number of potentially light coins is odd and the number of potentially heavy coins is odd, and the total number of coins is not more than five. In this case the algorithm requires us to put aside two coins: one potentially heavy and one potentially light, but the put-aside pile cannot have more than one coin.

After checking small cases, we see that we cannot resolve the problem in one minute when there are 2 coins of different potential, or when the 4 coins are distributed as 1 and 3. On the other hand, if there are extra coins that are known to be real, then the above cases can be resolved. This means that the small cases are only a problem if they happen in the first minute. Hence, 

\begin{lemma}
Any number of coins $N > 4$ with known potential can be resolved in $\lceil \log_5 N \rceil$ minutes.
\end{lemma}

\section{Unlimited Supply of Real Coins}\label{sec:unlimited}

We say that the coin that is potentially light or potentially heavy has \textit{known potential}. The notion of known potential is important in solving Knop's puzzle and many other coin weighing puzzles due to the following theorem:

\begin{theorem}
In a coin-weighing puzzle, where only one coin is fake, any coin that visited the scales is either genuine or its potential is known.
\end{theorem}

\begin{proof}
If the scale ever balanced, all coins that were on it on any such occasion are real. Any coin that appeared on both a heavier pan
and a lighter pan is also real. Otherwise, the coins that only visited lighter pans are potentially light and the coins that only visited heavier pans are potentially heavy.
\end{proof}

Now let us go back to the original problem, in which we do not know the coins' potential at the start. Let us temporarily add an additional assumption to the original problem. Suppose there is an unlimited supply of coins that we know to be real. Let $u(n)$ be the maximum number of coins we can process in $n$ minutes if we do not know their potential and have an unlimited supply of real coins.

\begin{lemma}
$u(n) = 2\cdot 5^{n-1} + u(n-1).$
\end{lemma}

\begin{proof}
What information do we get after the first minute? Both scales might be balanced, meaning that the fake coin is in the leftover pile of coins with unknown potential. So we have to leave out not more than $u(n-1)$ coins. On the other hand, exactly one scale might be unbalanced. In this case, all the coins on this scale will have their potential revealed. The number of these coins cannot be more than $5^{n-1}$, so $u(n) \leq 2 \cdot 5^{n-1} + u(n-1)$. Can we achieve this bound? Yes. On each scale, put $5^{n-1}$ unknown coins on one pan, and $5^{n-1}$ real coins from the supply on the other. Thus,  $u(n) = 2 \cdot 5^{n-1} + u(n-1)$.
\end{proof}

We also can see that $u(1) = 3$. Indeed, on each scale put one coin against one real coin and have one coin in the leftover pile. Thus, the corollary:

\begin{corollary}
$u(n) = (5^n+1)/2.$
\end{corollary}

Thus, the answer to the puzzle problem with the additional resource of an unlimited supply of real coins is $(5^n+1)/2$. Clearly the answer without the additional resource cannot be larger. But what is it?

We assumed that there is an unlimited supply of real coins. But how many extra coins do we really need? The extra coins are needed for the first minute only, because after the first minute at least one of the scales will balance and many coins will be determined to be real. In the first minute, we need to put $5^{n-1}$ coins from the unknown pile on each scale. The coins do not have to be on the same pan. The only problem is that the number of coins is odd, so we need one extra real coin to make this number even. So our unlimited supply need not be unlimited---we just need two extra coins, one for each scale.

\section{The Puzzle Solution}\label{sec:solution}

Now let us go back and remember that the formula for $u(n)$ assumes an unlimited supply of real coins. The unlimited supply need not be more than two real coins. 

So, how can we solve the original problem? We already know that the only adjustment that is needed is in the first minute. In the first minute we put unknown coins against unknown coins, not more than $5^{n-1}$ on each scale, and, since the number on each scale must be even, the best we can do is put $5^{n-1}-1$ coins on each scale. Thus, the answer to the puzzle is $(5^n-3)/2$.

Do not forget that we cannot find the fake coin out of 2 coins, ever.

\begin{theorem}
Given two scales in parallel, the number of coins $N$ that can be optimally resolved in exactly $n$ minutes is: $(5^{n-1}-3)/2 \leq N < (5^n-3)/2$, with one exception: $N=2$, for which the fake coin cannot be identified.
\end{theorem}

Going back to the original puzzle: the largest number of coins that can be resolved in 5 minutes is 1561.

\section{More Scales}\label{sec:morescales}

It is straightforward to generalize the just-find problem to any number of scales used in parallel. Suppose the number of scales is $k$. The following problems can be solved in $n$ minutes:

\begin{description}
\item[Known Potential.] If all the coins have known potential, then any number of coins up to $(2k+1)^n$ can be resolved.
\item[Unlimited Supply of Real Coins.] If we do not know the potential of any coin and there is an unlimited supply of real coins, the maximum number of coins that can be solved is defined by a recursion: $u_k(n) = k \cdot (2k+1)^{n-1} + u_k(n-1)$ and $u_k(1)=k + 1$. Any number of coins up to $((2k+1)^n+1)/2$ can be resolved.
\item[General Case.] If we do not know the potential of any coin and there are no extra real coins, then any number of coins between 3 and $u_k(n) - k = ((2k+1)^n+1)/2 - k$ can be resolved.
\end{description}

Let me draw your attention to the fact that if $k=1$, then the general case is the classic problem of just finding the fake coin. So plugging in $k=1$ into the formula $((2k+1)^n+1)/2 - k$ above should give the answer given in Section~\ref{ancient}: $(3^n+1)/2-1$. In particular, for $n=3$, it should be 13. And it is.

\section{Find and Label}\label{sec:findlabel}

The methods described above can be used to answer another common question in the same setting: Find the fake coin and say whether it is heavier or lighter. If all coins have known potential, then the just-find problem is equivalent to the find-and-label problem.

The find-and-label problem can be solved by similar methods to the just-find problem. Namely, let us denote by $U_k(n)$ the number of coins that can be resolved in $n$ minutes in parallel on $k$ scales when there is an unlimited supply of extra real coins. Then the recursion is the same as for the just-find problem: $U_k(n) = k \cdot (2k+1)^{n-1} + U_k(n-1)$. The difference is in the starting point: and $U_k(1)=k$.

Similarly, if we do not have an unlimited supply of real coins, then the bound is described by the following theorem:

\begin{theorem}
There are $N$ coins one of which is fake, and it is not known whether it is heavier or lighter. There are also $k$ balance scales that can be used in parallel, one weighing per one minute. The maximum number of coins that requires $n$ minutes to find and label the fake coin is $((2k+1)^n+1)/2 - k -1$. If $N=2$, then the problem cannot be resolved.
\end{theorem}

Again, if $k=1$, then this is the classic problem of just funding and labeling the fake coin. So plugging in $k=1$ into the formula $((2k+1)^n+1)/2 - k-1$ above should give the answer from Section~\ref{ancient}: $(3^n+1)/2-2$. In particular, for $n=3$, it should be 12. And it is.

\section{Lazy Coin}\label{sec:lazy}

You might have noticed that the answer for the just-find and for the find-and-label problems differ by one:

\begin{lemma}
In the parallel weighing problem with one fake coin, the maximum number of coins that can be optimally resolved in $n$ weighings for the just-find problem is one more than the maximum number of coins that can be optimally resolved in the find-and-label problem in the same number of weighings. 
\end{lemma}

We already proved the lemma by explicitly calculating the answer. It would be nice if there was a simple argument to prove it without calculations. And such an argument exists if we restrict ourselves to static strategies. In a \textit{static} or \textit{non-adaptive} strategy you decide beforehand what your weighings are. Then, seeing the results of the weighings you can find the fake coin and label it if needed.

In a static strategy for the find-and-label problem every coin has to visit the scales at some point, or, otherwise, if the coin does not visit the scales and it happens to be fake,  it can not be labeled. In a static strategy for the just-find problem if every coin visits the scale, then all the coins can be labeled. Suppose we add an extra coin that is never on the scales. If all the weighings balance, then the extra coin is the fake one. We can not have two such coins. Indeed, if one of them is fake we can not differentiate between them. If all the coins visit the scales in the just-find problem, then all the coins can be labeled at the end.

Let us call a strategy that resolves the maximum number of coins in a given number of weighings a \textit{maximal} strategy. We just showed that a maximal static strategy for the just-find problem has to have a coin that does not go on the scales. Therefore, there is a bijection between maximal static strategies for the just-find and the find-and-label problems. The strategies differ by an extra coin that sits lazily outside the scales all the time.

In \textit{dynamic} or \textit{adaptive} strategies the next weighing depends on the results of the previous weighings. With dynamic strategies the story is more complicated. There is no a bijection any more.

On one hand it is possible to add a lazy coin to a strategy in the find-and-label problem to get a strategy in the just-find problem. But there exist maximal strategies in the just-find problem where all the coins can end up on the scale.

For example, consider the following strategy to just-find the fake coin out of 4 coins in 2 weighings on one scale. In the first weighing we balance the first coin against the second. If the weighing unbalances, we know that one of the participating coins is fake, and we know the potential of every participating coin. Then in the second weighing we balance the first and the second coins against the third and the fourth. In this example, all coins might visit the scale and 4 is the maximum number of coins that can be processed in 2 weighings.

Alas! There is no simple argument, but at least it is easy to remember, that the maximal strategies for these two problems differ by one coin.

\section{Acknowledgements}

I am grateful to Daniel Klain and Alexey Radul for helpful discussions.

\end{document}